\documentclass{amsart}

\usepackage{amssymb,color}
\usepackage{pdfsync}
\usepackage{amsmath}
\usepackage{amsthm}
\usepackage[mathscr]{eucal}
\usepackage{mathrsfs}

\newtheorem{thrm}{Theorem}[section]
\newtheorem{lemma}[thrm]{Lemma}
\newtheorem{prop}[thrm]{Proposition}
\newtheorem{cor}[thrm]{Corollary}
\newtheorem*{main}{Main Theorem}

\theoremstyle{definition}

\theoremstyle{remark}

\numberwithin{equation}{section}

\begin{document}

\bibliographystyle{plain}

\title[Projection operators on worm domains]
 {Regularity of projection operators attached to worm domains}

\author[D. E. Barrett, D. Ehsani, M. M. Peloso]
 {David E. Barrett, Dariush Ehsani, Marco M. Peloso}
\thanks{Work of the first author supported in part
by the National Science Foundation under Grant No. 1161735.}
\thanks{Work of the second author was (partially) supported by the Deutsche Forschungsgemeinschaft (DFG, German Research Foundation),
grant RU 1474/2 within DFG's Emmy Noether Programme.}

\address{Department of Mathematics\\
University of Michigan - Ann Arbor\\
2074 East Hall\\
Ann Arbor, Michigan  48109}
 \email{barrett@umich.edu}

\address{
Hochschule Merseburg\\
Eberhard-Leibnitz-Str. 2\\
 D-06217 Merseburg\\
 Germany}
 \email{dehsani.math@gmail.com}

\address{Dipartimento di Matematica\\
Universit\`a degli Studi di Milano\\
Via C. Saldini 50\\
I-20133 Milano} \email{marco.peloso@unimi.it}



\begin{abstract}
 We construct a projection operator
on an unbounded worm domain which maps subspaces of $W^s$ to
themselves.  The subspaces are determined by a Fourier
decomposition of $W^s$ according to a rotational invariance of
 the worm domain.
\end{abstract}

\maketitle

\section*{Introduction}

Our work is on the non-smooth unbounded worm domains
\begin{equation*}
D_{\beta}=
 \{ (z_1,z_2)\in \mathbb{C}^2:\mbox{Re}\left(z_1 e^{-i\log
 z_2\overline{z}_2}\right) >0,
  |\log
 z_2\overline{z}_2|<\beta-\pi/2 \} \qquad \beta>\pi/2.
\end{equation*}
 On a bounded version of the
domains $D_{\beta}$, given by
\begin{equation*}
 \Omega_c = \left\{(z_1,z_2): \left| z_1+e^{i\log
 z_2\overline{z}_2}\right|^2 <1, |\log
 z_2\overline{z}_2|<\beta-\pi/2\right\},
\end{equation*}
C. Kiselman showed the failure of the Bergman projection to
 preserve $C^{\infty}(\overline{\Omega}_c)$ \cite{Ki91}.
The model domains, $D_{\beta}$, were important
 in \cite{Ba92}, where the first author used them to show the Diedrich-Forn\ae ss worm domains
  (constructed in \cite{DF77}) provide a counterexample
to regularity of the Bergman projection on a smoothly bounded
pseudoconvex domain. In a detailed analysis of the Bergman kernel,
Krantz and the third author, in \cite{KP081}, studied the $L^p$
mapping properties of the Bergman projection on $D_{\beta}$,
obtaining the exact range of values of $p$ for which the mapping
is bounded.

In this article we look at regularity in terms of Sobolev spaces.
  We denote by $W^s(D_{\beta})$ the space of functions whose
  derivatives of order $\le s$ are in $L^2(D_{\beta})$,
  {and by $W^s_{\mathscr{D}}(D_{\beta})$ the closure of $\mathscr{D}:=C^{\infty}_0(D_{\beta})$
   in $W^s(D_{\beta})$}.  The first
  author's results on smooth domains
  relied on the fact (proved in the same
  paper)
    that the Bergman projection on the model domain,
    $D_{\beta}$, fails to map
   $W^s_{\mathscr{D}}(D_{\beta})$  to $W^s(D_{\beta})$ for large enough $s$
  \cite{Ba92}. More precisely, the failure to preserve Sobolev
  spaces was proved on subspaces (defined as $W^s_j(D_{\beta})$
  below).  This was instrumental in proving that
 Condition $R$, in which for each $s\ge0$ there exists an $M\ge 0$
  such that the Bergman projection is bounded as a map from
   $W^{s+M}_{\mathscr{D}}(\Omega)$ to $W^{s}(\Omega)$ (\cite{BL80}),
   fails when $\Omega$ is the Diederich-Forn\ae ss worm domain
   \cite{Ch96}.  We point out that
     for a smoothly bounded pseudoconvex domain Condition $R$
    is equivalent to the apparently stronger condition
     in which the larger domain $W^{s+M}(\Omega)$ replaces
 $W^{s+M}_{\mathscr{D}}(\Omega)$.
 This equivalence also holds on the domains $D_{\beta}$:
  the first author constructed a composition of first order operators which allow
  us to consider the Bergman projection acting on functions which
  vanish to desired order at the boundary, without changing the
  resulting image of the projection (see Theorem 2.2 and specifically
  Theorem 3.1 in \cite{Ba95}).

  The question remained whether there exists another
  (oblique) projection operator which preserves the level of the
   Sobolev spaces.  We construct such an operator in the present article.

We now state our main
 result.
From the rotational invariance of $D_{\beta}$ with respect to the
rotations, $\rho_{\theta}(z)=(z_1,e^{i\theta}z_2)$, we can
decompose the Bergman space $B(D_{\beta})=L^2(D_{\beta}) \cap
\mathcal{O}(D_{\beta})$ by
\begin{equation*}
 B(D_{\beta})=\bigoplus_{j\in \mathbb{Z}} B_{j}(D_{\beta}),
\end{equation*}
where $B_{j}(D_{\beta})$ consists of functions $f\in B(D_{\beta})$
satisfying $ f \circ \rho_{\theta} \equiv e^{ij\theta} f$.  The
space $L^2(D_{\beta})$ admits a similar decomposition into
subspaces $L^2_j(D_{\beta})$, and we can define
 $W^s_j(D_{\beta})=L^2_j(D_{\beta}) \cap  W^s(D_{\beta})$.

 Our main theorem is grounded on adjustments to
  factors which imply the obstruction to regularity of the Bergman
  projection on worm domains.  The Bergman kernel for each space,
  $B_j(D_{\beta})$
  is explicitly calculated and expressed as an
  integral in the form:
\begin{align*}
 K_{j}(z,w)
=&
  \frac{1}{2\pi^2} z_2^j\overline{w}_2^j
   \int_{\mathbb{R}} \frac{(\xi-\frac{j+1}{2})\xi}
   {\sinh\left[(2\beta-\pi)(\xi-\frac{j+1}{2})\right] \sinh \pi\xi}
 z_1^{i\xi-1}\overline{w}_1^{-i\xi-1} d\xi
 ,
\end{align*}
where, with an abuse of notation, we write
\begin{equation*}
 z_1^{\alpha}=(z_1 e^{-i\log z_2 \overline{z}_2})^{\alpha}
  e^{i\alpha \log z_2 \overline{z}_2}.
\end{equation*}
Such a power of $z_1$ is holomorphic on $D_{\beta}$ as is easy to
see, and it is locally constant in $|z_2|$, but not constant if
$\alpha$ is not an integer and $\beta>\pi$.  In fact, in this
case, the fiber over $z_1$ is a union of disjoint annuli in $z_2$
and the function is constant on each such annulus, but not
globally constant.

 Using the residue calculus, one can compute an asymptotic
expansion of the kernel (see \cite{Ba92}).  The poles
corresponding to non-integer multiples of $i$ of the kernel lead
to non-integer powers of $z_1$ and $w_1$ which ultimately lead to
the obstruction of regularity of the operator.

We construct a kernel which, when added to the Bergman kernel,
eliminates all such poles, and in this way
  we successfully remove the
  obstruction to regularity of the Bergman projection on the model
  domains, $D_{\beta}$, and
  construct new projections which preserve the level of
   Sobolev spaces:
\begin{main}
\label{mainproj}
 Let $\beta>\pi/2$, and $D_{\beta}$ be defined as above.
  For all $ j\in \mathbb{Z}$ there exists a bounded linear
  projection
\begin{equation*}
{\bf T}_j: L^2(D_{\beta}) \rightarrow B_{j}(D_{\beta})
\end{equation*}
which satisfies
\begin{equation*}
{\bf T}_j: W^s_{\mathscr{D}}(D_{\beta}) \rightarrow
W^s_j(D_{\beta})
\end{equation*}
for every $s\ge0$.
\end{main}

Much of this paper was discussed at collaborative meetings made
possible through invitations extended by the University of
Michigan and the Universit\`a degli Studi di Milano.  All authors
gratefully acknowledge the support from these institutions.
  We also thank the referee for a careful reading of
 the article as well as for many helpful suggestions in addition to
   pointing out an error in the
 calculation of the adjoint to the tangential operator,
  $\Lambda_t$, in Section \ref{propproj} which led to our use of the
  $W^s_{\mathscr{D}}(D_{\beta})$ Sobolev spaces.

\section{The Bergman projection on $D_{\beta}$}
\label{berg}

Following \cite{Ba92}, we introduce the domains
\begin{equation*}
D_{\beta}'=
 \Big\{ (z_1,z_2)\in \mathbb{C}^2:|\mbox{Im }z_1-
 \log z_2\overline{z}_2| <\pi/2,
  |\log
 z_2\overline{z}_2|<\beta-\pi/2 \Big\}
\end{equation*}
to aid in our study of the Bergman kernels on $D_{\beta}$.
$D_{\beta}'$ is related to $D_{\beta}$ via the biholomorphic
mapping
\begin{align}
 \label{psidef}
&\Psi : D_{\beta}' \rightarrow D_{\beta}\\
 \nonumber
&(z_1,z_2) \mapsto (e^{z_1},z_2).
\end{align}

Let $K_{D_{\beta}}(z,w)$ be the Bergman kernel for $D_{\beta}$,
and $K_{j}(z,w)$ the reproducing kernel for $B_{j}(D_{\beta})$; we
have the relation
 \begin{equation*}
K_{D_{\beta}}(z,w)=\sum_j K_j(z,w).
\end{equation*}
 We calculate $K_{j}(z,w)$ using Fourier transforms as in
\cite{Ba92}.

Let $S_{\beta}$ denote the strip
\begin{equation*}
 S_{\beta}:=\{z=x+iy\in\mathbb{C}:|y|<\beta \},
\end{equation*}
and let $\omega_j(y)$ be the continuous bounded function on the
interval
 $I_{\beta}:=\{ y:|y|<\beta\}$,
given by
$\omega_j=\pi\left(e^{(j+1)(\cdot)}\chi_{\beta-\pi/2}\right)\ast\chi_{\pi/2}$,
where for $a>0$, $\chi_a:=\chi_{(-a,a)}$, the characteristic
function of the interval $(-a,a)$.  We denote by
$\|\cdot\|_{\omega_j}$ the $L^2(S_{\beta})$-norm weighted with the
function $\omega_j$:
\begin{equation*}
\|f\|_{\omega_j}:=
 \left(
 \int_{S_{\beta}} |f(x,y)|^2 \omega_j(y) dx dy
\right)^{1/2}.
\end{equation*}
 We further define the weighted
 Bergman spaces on the strip $S_{\beta}$ by
\begin{equation*}
 B_{\omega_j}= \{ f \mbox{ holomorphic on } S_{\beta}
  : \| f\|_{\omega_j}^2<\infty \}.
\end{equation*}

For $f\in B_{\omega_j}$,
\begin{equation*}
\hat{f}(\xi,y)=\int_{\mathbb{R}} f(x+iy)e^{-ix\xi} dx
\end{equation*}
satisfies
 \begin{equation}
\label{hatf}
  \hat{f}(\xi,y)=e^{-y\xi}\hat{f}_{\mathbb{R}}(\xi),
\end{equation}
 where $\hat{f}_{\mathbb{R}}(\xi) := \hat{f}(\xi,0)$.
 %

Here and throughout we use the notation for complex variables
\begin{align*}
& z_1=x+iy\\
&w_1=x'+iy'.
\end{align*}
Define
\begin{equation*}
k_{j}'(\xi,y,w_1)=
 \frac{1}{\hat{\omega}_j(-2i\xi)} e^{i\xi(y-\overline{w}_1)},
\end{equation*}
 where $\hat{\omega}_j$ refers to the Fourier-Laplace transform of
 $\omega_j$, and satisfies
\begin{equation}
\label{hato}
 \hat{\omega}_j(-2i\xi)
 =
  \pi \frac
   {\sinh\left[(2\beta-\pi)(\xi-\frac{j+1}{2})\right] \sinh \pi\xi}
   {(\xi-\frac{j+1}{2})\xi}.
\end{equation}
We note that $\hat{\omega}_j$ extends to an entire
 function.
We claim that $k_{j}'$ corresponds to the kernel for the
 orthogonal projection
 on $D_{\beta}'$
according to the following lemma:
\begin{lemma}
 Let $K_{j}'(z,w)$ denote the reproducing kernel of the space
  $B_j(D_{\beta}')$.  Then
\label{equiv}
\begin{align*}
 K_{j}'(z,w)
=&
  \frac{1}{2\pi^2} z_2^j\overline{w}_2^j
   \int_{\mathbb{R}} \frac{(\xi-\frac{j+1}{2})\xi}
   {\sinh\left[(2\beta-\pi)(\xi-\frac{j+1}{2})\right] \sinh \pi\xi}
 e^{i(z_1-\overline{w}_1)\xi} d\xi
 .
\end{align*}
\end{lemma}
\begin{proof}
Let $\Gamma:\mathcal{B}_1\rightarrow \mathcal{B}_2$ be a
surjective isometry of two Bergman spaces.
 Let $K_1(z,w)$ be the reproducing kernel of the space
 $\mathcal{B}_1$ and $K_2(z,w)$ the kernel for
 $\mathcal{B}_2$.
 Then
 \begin{equation}
 \label{trnsfrm}
 K_2(z,w)=\overline{\Gamma_{w}\overline{\Gamma_z K_1(z,w)}}.
 \end{equation}

We now apply (\ref{trnsfrm}) to the spaces
 $\mathcal{B}_1=B_{\omega_j}$ and
  $\mathcal{B}_2=B_j(D_{\beta}')$.
From \cite{Ba92}
\begin{equation*}
 K_j(z_1,w_1)=\frac{1}{2\pi}\int_{\mathbb{R}}
  k_j'(\xi,y,w_1) e^{ix\xi} d\xi
\end{equation*}
is the reproducing kernel for $B_{\omega_j}$, and
\begin{align*}
\Gamma:B_{\omega_j} &\rightarrow B_j(D_{\beta}')\\
f(z_1)& \mapsto z_2^j f(z_1)
\end{align*}
is the isometry between Bergman spaces.
 Thus, by (\ref{trnsfrm})
\begin{equation*}
K_j'(z,w)=  \frac{1}{2\pi} z_2^j\overline{w}_2^j \int_{\mathbb{R}}
 k_{j}'(\xi,y,w_1) e^{ix\xi} d\xi
\end{equation*}
from which the lemma follows.
\end{proof}

\section{Improving the Bergman projection}
\label{correct}

Crucial to the proof in \cite{Ba92} of the failure of the Bergman
projection to
 preserve $W^s(D_{\beta})$ is the existence of poles of
 $k_{j}'(\xi,y,w_1)$ in the $\xi$ variable whose imaginary part
 is a non-integer multiple of $i$.  We see from (\ref{hato})
that such poles of $k_{j}'(\xi,y,w_1)$ are due to the zeros of
 $\hat{\omega}_j(-2i\xi)$ at
 $(j+1)/2+ ik\pi/(2\beta-\pi)$ for $k$ a non-zero integer.
 In this section we deal
 with this obstruction by adding a correction term
  which
 eliminates such poles.

We assume initially that $j=-1$.  To keep the notation that
 integral operators are defined by integrating functions against
 conjugates of functions of two variables (the kernel), we will work with terms in the kernel
  coming from $\hat{\omega}_j(2i\xi)$, observing that
 $\overline{\hat{\omega}_j(-2i\xi)}=\hat{\omega}_j(2i\xi)$.
 The goal in this section then is to find
a function, denoted by $\hat{h}(\xi,y)$, defined in
$\mathbb{C}\times I_{\beta}$ such that
 $\hat{h}(\xi,y)$,
 cancels the poles  of the function
\begin{equation}
 \label{bergpole}
\frac{1}{\hat{\omega}_{-1}(2i\xi)} e^{-\xi y}
\end{equation}
at $\xi=ik \nu_{\beta}$, for $k$ a non-zero integer, and
$\nu_{\beta}=\pi/(2\beta - \pi)$.  The function
 $\overline{\hat{h}(\xi,y)}$ will have
an inverse transform which is orthogonal to $B_{\omega_{-1}}$
 and satisfy
certain $L^2$ estimates which will be used in Section
\ref{mapping} to construct an integral operator.


 To ease notation
 we set
\begin{equation*}
\tau_{k}(\xi)=
(-1)^{k}\frac{e^{-k^2\nu_{\beta}^2}}{(2\beta-\pi)\pi}
 \frac{\xi^2}
 {\sinh(\pi \xi)}e^{-\xi^2}
\qquad k\in \mathbb{Z}.
\end{equation*}
We define
\begin{equation}
\label{hjk} \hat{h}_{k}(\xi,y)=
 \frac{\tau_{k}(\xi)e^{(\xi-2ik\nu_{\beta})y}}
 {\xi-ik\nu_{\beta}}.
\end{equation}

We note that the pole of (\ref{bergpole}) at $\xi=ik \nu_{\beta}$,
for $k$ a non-zero integer is the same as
 the pole of
  $\hat{h}_k$.
Our aim is to sum $\hat{h}_{k}$ over $k$ in order to
 produce a function which will be used to
  eliminate all such poles
   of
  (\ref{bergpole}).  The following proposition shows that
  we can sum over $k$.

To keep track of the poles, we introduce the set $P$ of all poles:
\begin{equation*}
P:=\{ ik\nu_{\beta} :k\ne0\}\cup \{ ik :k\ne0\}.
\end{equation*}

\begin{prop} \label{hj}
 Let $\hat{h}_{k}(\xi,y)$ be defined as above.
The infinite sum
\begin{equation*}
\sum_{k\ne 0} \hat{h}_{k}(\xi,\cdot)
\end{equation*}
converges in $L^{\infty}(I_{\beta})$ to a function
 $\hat{h}(\xi,\cdot)$ uniformly in $\xi$ on compact subsets of
 $\mathbb{C}\setminus P
 $.

Let $B_r=\cup B(ik\nu_{\beta};r)$ denote the union of balls
centered at elements of $P$ for some fixed radius $r>0$.  Let $U$
be any neighborhood of $P$ containing $B_r$.
 Then on $\mathbb{C}\setminus U \times I_{\beta}$
\begin{equation}
 \label{hunif}
|\hat{h}(\xi,y)|
 \lesssim
 |\xi|^{2}e^{- Re \xi^2}e^{(\beta-\pi)|Re\xi|},
\end{equation}
with the constant of inequality depending only on $U$.

\end{prop}
\begin{proof}
\begin{equation*}
\sum_{k\ne 0} \hat{h}_{k}(\xi,y)=\sum_{k\ne 0}
\frac{\tau_{k}(\xi)e^{(\xi-2ik\nu_{\beta})y}}
 {\xi-ik\nu_{\beta}}
\end{equation*}
is a sum of terms of the form
\begin{equation*}
e^{\xi y} \sum_{k\ne0}
 a_{k}(\xi)
\end{equation*}
where
\begin{equation*}
|a_{k}(\xi)|\lesssim \frac{1}{k}e^{-k^2\nu_{\beta}^2}
 |\xi|^{2}e^{- Re\xi^2} e^{-\pi|Re\xi|} \qquad
k \ne 0.
\end{equation*}

Inequality (\ref{hunif}) is then a consequence of
 \begin{align*}
 \nonumber
|\hat{h}(\xi,y)|&=
 \left| e^{\xi y}
\sum_{k}
 a_{k}(\xi) e^{-2ik\nu_{\beta} y}
 \right|\\
 \nonumber &\lesssim
  e^{ \beta|Re\xi|}
 \sum_{k}
 |a_{k}(\xi)|
 \\
&\lesssim
 |\xi|^{2}e^{- Re \xi^2}e^{(\beta-\pi)|Re\xi|}.
 \end{align*}
\end{proof}

 We note for $f\in B_{\omega_{-1}}$:
\begin{align*}
\int_{\mathbb{R}}\int_{I_{\beta}}
  \overline{\hat{h}_{k}(\xi,y)}   \hat{f}(\xi,y)\omega_{-1}(y)dy
  d\xi&=\int_{\mathbb{R}}\int_{I_{\beta}}
  \overline{\hat{h}_{k}(\xi,y)}  e^{-y\xi}\hat{f}_{\mathbb{R}}(\xi)
   \omega_{-1}(y) dy
 d\xi\\
 & =
\int_{\mathbb{R}}\frac{\tau_{k}(\xi)}
 {\xi+ik\nu_{\beta}}\hat{f}_{\mathbb{R}}(\xi)
 \left[
  \int_{I_{\beta}} e^{2ik\nu_{\beta}y} \omega_{-1}(y)
    dy \right]
 d\xi\\
 &=0
 ,
\end{align*}
 where we use the representation of $f$ in (\ref{hatf})
  in the first step, and the fact that
$\int_{I_{\beta}} e^{2ik\nu_{\beta}y} \omega_{-1}(y)
    dy=\hat{\omega}_{-1}(-2k\nu_{\beta})=0$ in the last.

We collect the essential properties, which follow directly from
the above, of the kernel function
 $h(x,y)$ in the following theorem:
\begin{thrm}
 \label{proph}
There exists $h(x,y) \in L^2_{\omega_{-1}}(S_{\beta})$ with the
following properties:
\\
$(i)$  For each $y\in I_{\beta}$, the poles of
\begin{equation*}
\hat{h}(\xi,y)+\frac{1}{\hat{\omega}_{-1}(2i\xi)} e^{-\xi y}
\end{equation*}
with respect to $\xi$ lie at only integer multiples of
 $i$.
\\
$(ii)$  The kernel given by
\begin{equation*}
\mathcal{H}'(z,w)=\frac{1}{2\pi}
 \frac{1}{z_2 \overline{w}_2}
 \int_{\mathbb{R}}
   \overline{\hat{h}(\xi,y)}e^{i( x-
    \overline{w}_1)\xi} d\xi
\end{equation*}
is orthogonal to the space  $B_{-1}(D_{\beta}')$ in the sense that
 $\mathcal{H}'(\cdot,w) \perp B_{-1}(D_{\beta}')$.
\\
$(iii)$  Let $U$ be any neighborhood of $P$ containing $B_r$ for
some $r>0$.
 Then on $\mathbb{C}\setminus U \times I_{\beta}$
\begin{equation*}
|\hat{h}(\xi,y)|
 \lesssim
 |\xi|^{2}e^{- Re \xi^2}e^{(\beta-\pi)|Re\xi|},
\end{equation*}
with the constant of inequality depending only on $U$.
\end{thrm}

We also denote the horizontal lines
\begin{equation*}
S_t=\{ \mathbb{R}+it\}
\end{equation*}
for $t\in \mathbb{R}$. From the Theorem \ref{proph} $iii)$, we
have in particular,
  on any given $S_t$
 such that $S_t\cap P = \emptyset$,  $\hat{h}(\xi,y)$
satisfies the following estimates
 uniformly, i.e. with constant of inequality independent of $\xi$:
\begin{equation}
\label{hjest}
 \int_{I_{\beta}}
  \left| \hat{h}(\xi,y)\right|^2
  dy \lesssim
  |\xi|^{4}e^{-2 Re \xi^2}e^{2(\beta-\pi)|Re\xi|}.
\end{equation}

\section{Mapping properties}
 \label{mapping}
We begin this section with some integral estimates for our
constructed correction term.   We let $\mathcal{H}'(z,w)$ be as in
Theorem \ref{proph}.  Due to the $\overline{z}_2^{-1}$ factor in
 $\overline{\mathcal{H}'(z,w)}$, the operator determined by the
 kernel, $\mathcal{H}'(z,w)$, will have its action restricted to
 the $L^2_{-1}(D_{\beta}')$ component of a given function in
  $L^2(D_{\beta}')$

We use the equivalence between Bergman spaces given in Lemma
\ref{equiv} in the proof of the next proposition: for $G\in
B_{-1}(D_{\beta}')$, $G$ is of the form $G=g(z_1) z_2^{-1}$, where
$g\in B_{\omega_{-1}}$, and $\|G\|_{B_{-1}(D_{\beta}')}=
\|g\|_{B_{\omega_{-1}}}$.
\begin{prop}
\label{mappingprop}
 Let $\beta>\pi/2$, and ${\bf H}'$ be the integral operator
\begin{equation*}
{\bf H'} f(w) = \int_{D'_\beta} f(z) \overline{\mathcal{H}'(z,w)}
    dV(z) ,
\end{equation*}
where
\begin{equation*}
\mathcal{H}'(z,w)= \frac{1}{2\pi}
 \frac{1}{z_2\overline{w}_2}
 \int_{\mathbb{R}}
   \overline{\hat{h}(\xi,y)}e^{i( x-
    \overline{w}_1)\xi} d\xi.
\end{equation*}
 Then
\begin{equation*}
{\bf H'} : L^2(D_{\beta}') \rightarrow B_{-1}(D_{\beta}'),
\end{equation*}
and
\begin{equation*}
\| {\bf H}' f\|_{B_{-1}(D_{\beta}')}
 \lesssim \|f\|_{L_{-1}^2(D_{\beta}')}.
\end{equation*}
\end{prop}
\begin{proof}
We write $D_{\beta}'=\mathbb{R}\times d_{\beta}'$,
 where
\begin{equation*}
d_{\beta}'=
 \{ (y,z_2)\in \mathbb{R}\times\mathbb{C}:|y-
 \log z_2\overline{z}_2| <\pi/2,
  |\log
 z_2\overline{z}_2|<\beta-\pi/2 \}.
\end{equation*}
 Then,
\begin{align*}
 {\bf H}' f(w)
 &= \frac{1}{2\pi}\frac{1}{w_2} \int_{D_{\beta}'}
\frac{1}{\overline{z}_2}
 \int_{\mathbb{R}}
   \hat{h}(\xi,y)e^{-i( x-
    w_1)\xi} d\xi f(z) dV(z) \\
 &=\frac{1}{2\pi}\frac{1}{ w_2} \int_{d_{\beta}'}
\frac{1}{\overline{z}_2}
 \int_{\mathbb{R}} \int_{\mathbb{R}}
   \hat{h}(\xi,y)e^{-i( x-
    w_1)\xi}f(x,y,z_2) dx d\xi dy dV(z_2)\\
    &=\frac{1}{2\pi}\frac{1}{ w_2} \int_{d_{\beta}'}
\frac{1}{\overline{z}_2}
 \int_{\mathbb{R}}
   \hat{h}(\xi,y)e^{i w_1\xi}
   \hat{f}(\xi,y,z_2) d\xi dy dV(z_2).
\end{align*}

We use a decomposition of $f$ according to
 \begin{equation*}
 f(z)=\sum_{j}f_j(z), \qquad f_{j}(z)\in L^2_{j}(D_{\beta}').
\end{equation*}
 Using the orthogonality of powers of $z_2$ (over circular
 regions) we can isolate any
  $f_j$ by integrating through $\overline{z}_2^j$.
This is used in the third step below where after integrating over
$z_2$ only $f_{-1}(z)$ terms remain:
\begin{align*}
 &\|{\bf H}' f \|^2_{B_{-1}(D_{\beta}')}
 =\frac{1}{4\pi^2}
\left\|
 \int_{d_{\beta}'}
\frac{1}{\overline{z}_2}
 \int_{\mathbb{R}}
   \hat{h}(\xi,y)e^{i (\cdot) \xi}
   \hat{f}(\xi,y,z_2) d\xi dy dV(z_2)
 \right\|^2_{B_{\omega_{-1}}}\\
&=\frac{1}{4\pi^2}
  \int_{I_{\beta}}
\int_{\mathbb{R}}\left|
  \int_{d_{\beta}'}
\frac{1}{\overline{z}_2}
 \int_{\mathbb{R}}
   \hat{h}(\xi,y)e^{-y'\xi}e^{i x' \xi}
   \hat{f}(\xi,y,z_2) d\xi dy dV(z_2) \right|^2
   dx'  \omega_{-1}(y')
    dy'\\
 &=\frac{1}{4\pi^2}
  \int_{I_{\beta}}
\int_{\mathbb{R}}\left|
  \int_{d_{\beta}'}
\frac{1}{\overline{z}_2}
   \hat{h}(\zeta,y)e^{-y'\zeta}
   \hat{f}_{-1}(\zeta,y,z_2)  dy dV(z_2) \right|^2
   d\zeta  \omega_{-1}(y')
    dy'\\
&\lesssim
  \int_{I_{\beta}}
\int_{\mathbb{R}}\Bigg[ \left(
  \int_{I_{\beta}}
  \left| \hat{h}(\zeta,y)\right|^2\omega_{-1}(y)dy \right)\times\\
  &\qquad\qquad
  \left( \int_{d_{\beta}'} \left|
   \hat{f}_{-1}(\zeta,y,z_2)\right|^2  dy dV(z_2)  \right)
  e^{-2y'\zeta} \Bigg] d\zeta   \omega_{-1}(y')
    dy'.
\end{align*}

From Theorem \ref{proph} $(iii)$ and (\ref{hjest}) we have that
\begin{equation*}
 \int_{I_{\beta}}
  \left| \hat{h}(\zeta,y)\right|^2\omega_{-1}(y)dy
   \lesssim
|\zeta|^{4}e^{-2 Re \zeta^2}e^{2(\beta-\pi)|Re\zeta|}.
\end{equation*}
We continue with our estimate of $ \| {\bf H}' f
\|_{B_{-1}(D_{\beta}')}$:
\begin{align*}
 \| {\bf H}' f \|_{B_{-1}(D_{\beta}')}^2
  & \lesssim
\kern-2pt\int_{\mathbb{R}}
  \int_{d_{\beta}'} \kern-2pt\left|
   \hat{f}_{-1}(\zeta,y,z_2)\right|^2  dy dV(z_2)
   |\zeta|^{4}e^{-2\zeta^2}
    e^{2(\beta-\pi) |\zeta|} \hat{\omega}_{-1}(-2i\zeta)
   d\zeta\\
&\lesssim \| f_{-1} \|^2_{L^2(D_{\beta}')},
\end{align*}
where the last estimate follows by the fact that
 the term $|\zeta|^{4}e^{-2\zeta^2}
    e^{2(\beta-\pi) |\zeta|} \hat{\omega}_{-1}(-2i\zeta)$ is bounded with
    respect to $\zeta$.
\end{proof}

 We recall the biholomorphic mapping $\Psi : D_{\beta}' \rightarrow D_{\beta}$
 from (\ref{psidef}).  Through a change of variables $\Psi^{-1}$,
  ${\bf H}'$ induces an integral operator on
   $L^2(D_{\beta})$:
    $g \mapsto (g\circ \Psi) det(\Psi^{-1})'$,
     $(\Psi^{-1})'$ being the complex Jacobian
     of $(\Psi^{-1})'$, is an isometry between
     $L^2(D_{\beta})$ and $L^2(D_{\beta}')$, and in fact since
    $\Psi$ is biholomorphic, between Bergman spaces (see also \eqref{trnsfrm}).
 In this regard,
 we define the kernel
\begin{equation}
 \label{kernh}
 \mathcal{H}(z,w) = \frac{1}{z_1\overline{w}_1}
  \mathcal{H}'(\Psi^{-1}z,\Psi^{-1}w),
\end{equation}
using the fact that $det (\Psi^{-1} (z) )' = \frac{1}{z_1}$.

  Let ${\bf H}$ be the integral operator
\begin{equation*}
{\bf H} f(w) = \int_{D_\beta} f(z) \overline{\mathcal{H}(z,w)}
    dV(z) ,
\end{equation*}
where $\mathcal{H}(z,w)$ is given by \eqref{kernh}.

 Then as a result of Proposition \ref{mappingprop}, we have the
 following
\begin{cor} We have that
\label{hjmap}
\begin{equation*}
{\bf H} : L^2(D_{\beta}) \rightarrow B_{-1}(D_{\beta}),
\end{equation*}
and
\begin{equation*}
\| {\bf H} f\|_{B_{-1}(D_{\beta})}
 \lesssim \|f\|_{L_{-1}^2(D_{\beta})}.
\end{equation*}
\end{cor}

 We now define the projection operator ${\bf T}_{-1}$ as
\begin{equation*}
{\bf T}_{-1}={\bf P}_{-1}+{\bf H},
\end{equation*}
 where ${\bf P}_{-1}:L^2(D_{\beta})\rightarrow B_{-1}(D_{\beta})$
is the orthogonal projection operator.

\section{Properties of the projection ${\bf T}_{-1}$}
 \label{propproj}

\begin{thrm}
 \label{projthrm}
  Let $\beta>\pi/2$ and ${\bf T}_{-1}={\bf P}_{-1}+{\bf H}$.
   Then
\begin{equation*}
{\bf T}_{-1}: L^2(D_{\beta}) \rightarrow B_{-1}(D_{\beta}).
\end{equation*}
Furthermore, ${\bf T}_{-1}$ is a projection, and has the
regularity property
\begin{equation}
 \label{wk}
{\bf T}_{-1}: W^k_{\mathscr{D}}(D_{\beta}) \rightarrow
W^k_{-1}(D_{\beta})
 \qquad \forall k,
\end{equation}
and
\begin{equation*}
\|{\bf T}_{-1}f\|_{W^k_{-1}(D_{\beta})}
 \lesssim \|f\|_{W^k(D_{\beta})}
\end{equation*}
  for $f\in W^k_{\mathscr{D}}(D_{\beta})$.
\end{thrm}
\begin{proof}

The mapping from $L^2(D_{\beta})$ to $B_{-1}(D_{\beta})$ follows
from the corresponding properties of ${\bf P}_{-1}$ and ${\bf H}$
(see Corollary \ref{hjmap}).

That ${\bf T}_{-1}$ is a projection follows from ${\bf P}_{-1}$
being a projection and from the restriction of ${\bf H}$ to
$B_{-1}(D_{\beta})$ being equivalently 0 (from Theorem \ref{proph}
$ii.$):
\begin{align*}
{\bf T}_{-1}^2&=
 {\bf P}_{-1}^2+{\bf P}_{-1}{\bf H}+{\bf H}{\bf P}_{-1}
  +{\bf H}^2\\
  &={\bf P}_{-1}+{\bf H}\\
  &={\bf T}_{-1}.
\end{align*}

Since ${\bf T}_{-1} f$ is holomorphic, to prove (\ref{wk}) we
estimate the $L^2$ norm of holomorphic derivatives of ${\bf
T}_{-1} f$.  Also, ${\bf T}_{-1} f$ is of the form $g(w_1,|w_2|)
w_2^{-1}$, where the function $g(w_1,|w_2|)$ is holomorphic and
 locally constant in $w_2$, so its derivatives in $w_2$ are zero and we only
need to estimate the derivatives with respect to the first
variable. To prove the theorem we thus show
\begin{equation}
\label{w1der}
 \left\| \frac{\partial^k}{\partial w_1^k} {\bf T}_{-1} f
\right\|_{L^2(D_{\beta})} \lesssim \|f_{-1}\|_{W^k(D_{\beta})},
\end{equation}
 for $f\in W^k_{\mathscr{D}}(D_{\beta})$.

The domain $D_{\beta}'$ is related to $D_{\beta}$ via the
biholomorphic mapping $\Psi$.
   We can then read off the
kernels attached to the domain $D_{\beta}$ from the transformation
formula applied to the corresponding kernels on $D_{\beta}'$, as
in (\ref{kernh}). We have
the relations
\begin{align*}
&K_{-1}(z,w) = \frac{1}{z_1\overline{w}_1}
  K_{-1}'(\Psi^{-1}z,\Psi^{-1}w)\\
&\mathcal{H}(z,w) = \frac{1}{z_1\overline{w}_1}
  \mathcal{H}'(\Psi^{-1}z,\Psi^{-1}w)\\
&\mathcal{T}_{-1}(z,w) = \frac{1}{z_1\overline{w}_1}
  \mathcal{T}_{-1}'(\Psi^{-1}z,\Psi^{-1}w)
  ,
\end{align*}
where $K_{-1}$, $\mathcal{H}$, $\mathcal{T}_{-1}$
 (resp. $K_{-1}'$, $\mathcal{H}'$, $\mathcal{T}_{-1}'$)
 are the kernels for, respectively, $ {\bf P}_{-1}$, ${\bf H}$,
 ${\bf T}_{-1}$ (resp. $ {\bf P}_{-1}'$, ${\bf H}'$,
 ${\bf T}_{-1}'$).

 Using integration by parts, we relate
$\frac{\partial^k}{\partial w_1^k} {\bf T}_{-1} f$ to
  $k^{th}$ order
 derivatives falling on $f$.

From above, we have
\begin{equation*}
{\bf T}_{-1} f (w) =
 \int_{D_{\beta}}
   \overline{\mathcal{T}_{-1}(z,w)} f(z) dV(z),
\end{equation*}
where
\begin{align*}
\overline{\mathcal{T}_{-1} (z,w)}
 = \frac{1}{2\pi}
 \frac{1}{\overline{z}_2 w_2}
 \int_{\mathbb{R}}&\Bigg(
   \frac{1}{\hat{\omega}_{-1}(2i\xi)}
   \overline{z}_1^{-i\xi-1} w_1^{i\xi-1}\\
   &
 +\hat{h}(\xi,(\log z_1 -\log \overline{z}_1)/2i)
 z_1^{-i\xi/2-1}\overline{z}_1^{-i\xi/2}
 w_1^{i\xi-1}
  \Bigg)d\xi.
\end{align*}
 By virtue of the
 factor $\overline{z}_2^{-1}$ in $\overline{\mathcal{T}_{-1}
 (z,w)}$, all action is isolated on $f_{-1}(z)$.
 Thus,
\begin{align*}
{\bf T}_{-1} f (w) =&
 \int_{D_{\beta}}
   \overline{\mathcal{T}_{-1}(z,w)} f(z) dV(z)\\
   =& \int_{D_{\beta}}
   \overline{\mathcal{T}_{-1}(z,w)} f_{-1}(z) dV(z).
\end{align*}

 Furthermore,
\begin{equation}
\label{contribute}
 \frac{\partial^k}{\partial w_1^k} {\bf T}_{-1} f
 = \int_{D_{\beta}}
   \frac{\partial^k}{\partial w_1^k}
   \overline{\mathcal{T}_{-1}(z,w)} f_{-1}(z) dV(z),
\end{equation}
and
\begin{align}
\nonumber
 \frac{\partial^k}{\partial w_1^k}&
 \overline{\mathcal{T}_{-1}(z,w)}
 =\\
 \nonumber
  &\frac{1}{2\pi}
  \frac{1}{\overline{z}_2 w_2}
 \int_{\mathbb{R}}(i\xi-1)(i\xi-2)\cdots (i\xi-k) \Bigg(
   \frac{1}{\hat{\omega}_{-1}(2i\xi)}
   \overline{z}_1^{-i\xi-1}w_1^{i\xi-k-1}\\
 \label{dwtk}
   &\qquad
 +\hat{h}(\xi,(\log z_1 -\log \overline{z}_1)/2i)
 z_1^{-i\xi/2-1}\overline{z}_1^{-i\xi/2}w_1^{i\xi-k-1}\Bigg)
  d\xi.
\end{align}

Our strategy is roughly as follows: we use shifts of contours of
integration to write the integrands of \eqref{dwtk} using
derivatives with respect to $z_1$; we make sure Fubini's theorem
applies with respect to the $z$ and $\xi$ integrals and then we
take the $z_1$ derivatives outside the $\xi$ integrals; finally we
can then perform an integration by parts in the $z_1$ variable in
\eqref{contribute}.

When shifting the contour of integration, in order to verify that
Fubini's theorem applies, we work with the two cases, each of
which determines a different direction of shift:
\begin{enumerate}
\item[$i)$] $|w_1|<|z_1|$
 \item[$ii$)] $|z_1|<|w_1|$.
\end{enumerate}

   To illustrate the cases, we
consider integrals of the form
 \begin{equation*}
 \phi_t(w_1)= \int_{U} \frac{1}{\overline{z}_2}
\int_{Im(\xi)=t} \sigma_{w_1}(\xi,z_1,\overline{z}_1) f_{-1}(z)
  d\xi dV(z),
\end{equation*}
where $\sigma_{w_1}$ will be either
\begin{equation*}
 (i\xi-1)(i\xi-2)\cdots (i\xi-k)
   \frac{1}{\hat{\omega}_{-1}(2i\xi)}
   \frac{1}{z_1w_1^{k+1}}
   \left(\frac{w_1}{\overline{z}_1}\right)^{i\xi}
\end{equation*}
or
\begin{equation*}
 (i\xi-1)(i\xi-2)\cdots (i\xi-k)
\hat{h}(\xi,(\log z_1 -\log
\overline{z}_1)/2i)\frac{1}{z_1w_1^{k+1}}
\left(\frac{w_1}{|z_1|}\right)^{i\xi},
\end{equation*}
and the domain of integration $U$ will be either
 $D_{\beta}\bigcap \{|w_1|<|z_1|\}$
 or $D_{\beta}\bigcap \{|z_1|<|w_1|\}$. Using the
 estimates for $\hat{\omega}_{-1}(2i\xi)$ and
  the estimate in (\ref{hunif}) for
   $\hat{h}$, we have
\begin{equation}
\label{phit} |\phi_t(w_1)|
 \lesssim
  \int_{U} \left|\frac{1}{\overline{z}_2}
 \frac{1}{z_1w_1^{k+1}} \right|
 \left(\frac{|z_1|}{|w_1|}\right)^t |f_{-1}(z)|
   dV(z).
\end{equation}
We see Fubini's theorem applies in case $i)$ when $t<0$ and in
case $ii)$ when $t>0$.  The signs of $t$ correspond to shifts in
the lower- and upper half planes, respectively.

 We now proceed to the write
an expression for the kernel $\frac{\partial^k}{\partial w_1^k}
   \overline{\mathcal{T}_{-1}(z,w)}$ in terms of
    derivatives with respect to the $z$ variable,
    corresponding to the two cases.  It will be shown in both
    cases we are lead to the same expression.

Case $i)$.  By construction of the term $h$ in Section
\ref{correct} the integrand exhibits poles only at integer
multiples of $i$, of which those at $-i, -2i, \ldots, -ik$ are
cancelled.  We therefore deform the contour of integration in
(\ref{dwtk})
 to $\mathbb{R}-ik$.
 The contribution of the sides of the contour are null due to the
 exponential decay in $\xi$ of the integrand.

 We now work with the contour of integration in
(\ref{dwtk})
 deformed to $\mathbb{R}-ik$.  We first consider
\begin{align*}
 &\frac{1}{2\pi} \frac{1}{\overline{z}_2 w_2}
 \int_{\mathbb{R}-ik}
(i\xi-1)(i\xi-2)\cdots (i\xi-k)
  \frac{1}{\hat{\omega}_{-1}(2i\xi)}
\overline{z}_1^{-i\zeta-1} w_1^{i\xi-k-1} d\xi
 =
 \\
 & \frac{1}{2\pi} \frac{1}{\overline{z}_2 w_2}
 \int_{\mathbb{R}} (i\zeta+k-1)(i\zeta+k-2)\cdots (i\zeta)
   \frac{1}{\hat{\omega}_{-1}(2i(\zeta-ik))}
   \overline{z}_1^{-i\zeta-k-1}w_1^{i\zeta-1} d\zeta.
\end{align*}
 We use
\begin{equation*}
\frac{1}{\hat{\omega}_{-1}(2i(\zeta-ik))}=
 (-1)^k \frac{1}{\pi}
  \frac{(\zeta-ik)^2}
  {\sinh[(2\beta-\pi)(\zeta-ik)]
  \sinh(\pi\zeta)},
 \end{equation*}
 hence
\begin{align*}
   &\frac{(i\zeta+k-1)(i\zeta+k-2)\cdots (i\zeta)}
   {\hat{\omega}_j(2i(\zeta-ik))}
   \overline{z}_1^{-i\zeta-k-1}w_1^{i\zeta-1}\\
   &\qquad \qquad= (-1)^k \frac{1}{\pi}
   \frac{(\zeta-ik)(\zeta-ik)}
   {\sinh[(2\beta-\pi)(\zeta-ik)]
   \sinh(\pi\zeta)}\times\\
   &\qquad \qquad \qquad \qquad
   (i\zeta+k-1)(i\zeta+k-2)\cdots (i\zeta)
   \overline{z}_1^{-i\zeta-k-1}w_1^{i\zeta-1}\\
&\qquad \qquad = (-1)^k \frac{1}{\pi}
   \frac{(\zeta-ik)(i\zeta)}
   {\sinh[(2\beta-\pi)(\zeta-ik)]
   \sinh(\pi\zeta)}\times\\
   &\qquad \qquad \qquad \qquad
   (\zeta-ik)(i\zeta+k-1)\cdots (i\zeta+1)
   \overline{z}_1^{-i\zeta-k-1}w_1^{i\zeta-1}\\
&\qquad \qquad=  \frac{1}{\pi}
   \frac{(\zeta-ik)\zeta}
   {\sinh[(2\beta-\pi)(\zeta-ik)]
   \sinh(\pi\zeta)}
   \frac{\partial^k}{\partial {\overline{z}_1}^k}
   \overline{z}_1^{-i\zeta-1}w_1^{i\zeta-1}.
\end{align*}
For the integral above we thus have
\begin{align*}
\frac{1}{2\pi} \frac{1}{\overline{z}_2 w_2}
 &\int_{\mathbb{R}-ik}
(i\xi-1)(i\xi-2)\cdots (i\xi-k)
  \frac{1}{\hat{\omega}_{-1}(2i\xi)}
\overline{z}_1^{-i\zeta-1} w_1^{i\xi-k-1} d\xi
 =
 \\
& \frac{1}{2\pi^2} \frac{1}{\overline{z}_2 w_2}
 \int_{\mathbb{R}}
\frac{(\zeta-ik)\zeta}
   {\sinh[(2\beta-\pi)(\zeta-ik)]
   \sinh(\pi\zeta)}
   \frac{\partial^k}{\partial {\overline{z}_1}^k}
   \overline{z}_1^{-i\zeta-1}w_1^{i\zeta-1}d\zeta.
\end{align*}
Similarly, we work with
\begin{align}
 \label{hint}
\frac{1}{2\pi} \frac{1}{\overline{z}_2 w_2}
 \int_{\mathbb{R}-ik}
(i\xi-1)(i\xi-2)\cdots (i\xi-k)
  &\hat{h}(\xi,(\log  z_1 -\log \overline{z}_1)/2i)\times\\
  \nonumber
& z_1^{-i\xi/2-1}\overline{z}_1^{-i\xi/2}
  w_1^{i\xi-k-1}
  d\xi.
 \end{align}

Let us write
\begin{equation*}
 \label{h-g}
\hat{h}(\xi,(\log z_1 -\log \overline{z}_1)/2i)
 = \frac{\xi}{\sinh(\pi \xi)} g(\xi,z_1),
\end{equation*}
and note that $g(\xi,z_1)$ has the property
\begin{equation*}
\Lambda_tg(\xi,z_1) =0,
\end{equation*}
where
\begin{equation*}
\Lambda_t := \left ( \left(\frac{z_1}{\overline{z}_1}\right)^{1/2}
\frac{\partial}{\partial z_1}
 + \left(\frac{\overline{z}_1}{z_1}\right)^{1/2}
 \frac{\partial}{\partial
 \overline{z}_1} \right).
\end{equation*}

The integrand in (\ref{hint}) can thus be written according to
\begin{align*}
 (i\zeta+k-1)&(i\zeta+k-2)\cdots (i\zeta)
  \hat{h}(\zeta-ik,(\log z_1 -\log \overline{z}_1)/2i)\times\\
  &
 \qquad\qquad\qquad\qquad\qquad\qquad z_1^{(-i\zeta-k)/2-1}\overline{z}_1^{(-i\zeta-k)/2}
  w_1^{i\zeta-1}\\
  &=(-1)^k
 \frac{\zeta}{\sinh(\pi\zeta)}g(\zeta-ik,z_1)\times\\
 &\qquad
 (i\zeta+k)(i\zeta+k-1)\cdots (i\zeta+1)
 z_1^{(-i\zeta-k)/2-1}\overline{z}_1^{(-i\zeta-k)/2}
  w_1^{i\zeta-1}\\
  &=
 \frac{\zeta}{\sinh(\pi\zeta)}g(\zeta-ik,z_1)
  \times\\
 &\qquad
 \left ( \left(\frac{z_1}{\overline{z}_1}\right)^{1/2}
\frac{\partial}{\partial z_1}
 + \left(\frac{\overline{z}_1}{z_1}\right)^{1/2}
 \frac{\partial}{\partial
 \overline{z}_1} \right)^k
 z_1^{-i\zeta/2-1}\overline{z}_1^{-i\zeta/2}
  w_1^{i\zeta-1}\\
 & =
 \frac{\zeta}{\sinh(\pi\zeta)}g(\zeta-ik,z_1)
(\Lambda_t)^k z_1^{-i\zeta/2-1}\overline{z}_1^{-i\zeta/2}
  w_1^{i\zeta-1}
 .
\end{align*}

Therefore,
\begin{align*}
\int&_{|w_1|<|z_1|}
   \frac{\partial^k}{\partial w_1^k}
   \overline{\mathcal{T}_{-1}(z,w)} f(z) dV(z)
=-\frac{1}{2\pi}\frac{1}{\overline{z}_2 w_2}\times\\
 &\int_{Re w_1< Re z_1}\Bigg[\frac{1}{\pi}
 \int_{\mathbb{R}}
 \frac{(\zeta-ik)\zeta}
   {\sinh[(2\beta-\pi)(\zeta-ik)]
   \sinh(\pi\zeta)}
   \frac{\partial^k}{\partial {\overline{z}_1}^k}
   \overline{z}_1^{-i\zeta-1}w_1^{i\zeta-1} d\zeta+\\
   &
\int_{\mathbb{R}}
 \frac{\zeta}{\sinh(\pi\zeta)}g(\zeta-ik,z_1)
(\Lambda_t)^k z_1^{-i\zeta/2-1}\overline{z}_1^{-i\zeta/2}
  w_1^{i\zeta-1}
  d\zeta
   \Bigg] f_{-1}(z) dV(z).
\end{align*}

We remark that, as outlined above, the $\zeta$ and $z$
integrations can be switched (just consider the integral $\phi_k$
in \eqref{phit}).

Case $ii)$.  We begin by writing (\ref{dwtk}) in the form:

\begin{align*}
\frac{\partial^k}{\partial w_1^k}
 \overline{\mathcal{T}_{-1}(z,w)}
 &=
  \frac{1}{2\pi}\frac{1}{\overline{z}_2 w_2}\times\\
 \int_{\mathbb{R}}\Bigg[
    &\frac{(-1)^k}{\pi} \frac{(\xi+ik)(\xi)}
    {\sinh\left[(2\beta-\pi)(\xi)\right] \sinh (\pi\xi)}
   \frac{\partial^k}{\partial \overline{z}_1^k}
   \overline{z}_1^{-i\xi + (k-1)} w_1^{i\xi-1-k}+\\
    &(-1)^k\frac{\xi+ik}{\sinh(\pi\xi)}g(\xi,z_1)
 \left ( \Lambda_t \right)^k
  z_1^{-i\xi/2+k/2-1}\overline{z}_1^{-i\xi/2+k/2}
  w_1^{i\xi-k-1}\Bigg]d\xi,
\end{align*}
which is also obtained by
 deforming the contour of integration to
 $\mathbb{R}+ik$ (using that the sides of the contour give no contributions
 in the same manner as that of case $i)$) of the following integral
\begin{align*}
& -\frac{1}{2\pi}\int_{\mathbb{R}}\Bigg[\frac{1}{\pi}
 \frac{\zeta(\zeta-ik)}
    {\sinh\left[(2\beta-\pi)(\zeta-ik)\right] \sinh (\pi\zeta)}
   \frac{\partial^k}{\partial \overline{z}_1^k}
   \overline{z}_1^{-i\zeta-1} w_1^{i\zeta-1}\\
   &
\qquad\qquad +\frac{\zeta}{\sinh(\pi\zeta)}g(\zeta-ik,z_1)
 \left ( \Lambda_t \right)^k
 z_1^{-i\zeta/2-1}\overline{z}_1^{-i\zeta/2}
  w_1^{i\zeta-1}\Bigg]d\zeta,
\end{align*}
noting that the contribution from the poles at integer multiples
of $i$ are cancelled due to the differential operators.

Combining the results in cases $i)$ and $ii)$, we have
\begin{align*}
\nonumber \int_{D_{\beta}}
   &\frac{\partial^k}{\partial w_1^k}
   \overline{\mathcal{T}_{-1}(z,w)} f(z) dV(z)=\\
   \nonumber
& \int_{|w_1|<|z_1|}
   \frac{\partial^k}{\partial w_1^k}\overline{\mathcal{T}_{-1}(z,w)} f(z)
   dV(z)
+ \int_{|w_1|>|z_1|}
   \frac{\partial^k}{\partial w_1^k}\overline{\mathcal{T}_{-1}(z,w)} f(z)
   dV(z),
\end{align*}
where
\begin{align*}
\int_{|w_1|<|z_1|}
   &\frac{\partial^k}{\partial w_1^k}
   \overline{\mathcal{T}_{-1}(z,w)} f(z) dV(z)=\\
\nonumber
 -&\frac{1}{2\pi}
\int_{|w_1|<|z_1|} \frac{1}{\overline{z}_2 w_2}\times\\
\nonumber & \qquad
 \Bigg[\frac{1}{\pi}
 \int_{\mathbb{R}}
 \frac{(\zeta-ik)\zeta}
   {\sinh[(2\beta-\pi)(\zeta-ik)]
   \sinh(\pi\zeta)}
   \frac{\partial^k}{\partial {\overline{z}_1}^k}
   \overline{z}_1^{-i\zeta-1}w_1^{i\zeta-1} d\zeta+\\
\nonumber
   &
\int_{\mathbb{R}}
 \frac{\zeta}{\sinh(\pi\zeta)}g(\zeta-ik,z_1)
(\Lambda_t)^k z_1^{-i\zeta/2-1}\overline{z}_1^{-i\zeta/2}
  w_1^{i\zeta-1}
  d\zeta
   \Bigg] f_{-1}(z) dV(z)
\end{align*}
and
\begin{align*}
\int_{|w_1|>|z_1|}
   &\frac{\partial^k}{\partial w_1^k}
   \overline{\mathcal{T}_{-1}(z,w)} f(z) dV(z)=\\
\nonumber
 -&\frac{1}{2\pi}
\int_{|w_1|>|z_1|} \frac{1}{\overline{z}_2 w_2}\times\\
\nonumber & \qquad
 \Bigg[\frac{1}{\pi}
 \int_{\mathbb{R}}
 \frac{(\zeta-ik)\zeta}
   {\sinh[(2\beta-\pi)(\zeta-ik)]
   \sinh(\pi\zeta)}
   \frac{\partial^k}{\partial {\overline{z}_1}^k}
   \overline{z}_1^{-i\zeta-1}w_1^{i\zeta-1} d\zeta+\\
\nonumber
   &
\int_{\mathbb{R}}
 \frac{\zeta}{\sinh(\pi\zeta)}g(\zeta-ik,z_1)
(\Lambda_t)^k z_1^{-i\zeta/2-1}\overline{z}_1^{-i\zeta/2}
  w_1^{i\zeta-1}
  d\zeta
   \Bigg] f_{-1}(z) dV(z).
\end{align*}
We now use Fubini's theorem in both case $i)$ and $ii)$ to take
the derivatives outside of the $\zeta$ integrals, and then combine
the results above.  Before doing so, we note
 \begin{align*}
\Lambda_t=&\left ( \left(\frac{z_1}{\overline{z}_1}\right)^{1/2}
\frac{\partial}{\partial z_1}
 + \left(\frac{\overline{z}_1}{z_1}\right)^{1/2}
 \frac{\partial}{\partial
 \overline{z}_1} \right)\\
  =& \partial_{r_1},
\end{align*}
 where $r_1=|z_1|$,
is a tangential differential operator.
 To calculate the adjoint of $\Lambda_t$
   we note for fixed $z_2$, $z_1$ can be written with coordinates
 $t_1$ and $d_1$, $d_1$ representing the distance to the boundary
$\mbox{Re }z_1 e^{-i\log z_2\overline{z}_2}=0$, via
\begin{equation}
 \label{coortd}
 z_1=(t_1+id_1)e^{i\alpha}
\end{equation}
where $\alpha=\log|z_2|^2-\pi/2 $.  In these coordinates we
calculate
\begin{equation*}
 \Lambda_t =
  \frac{t_1}{\sqrt{t_1^2+d_1^2}} \frac{\partial}{\partial t_1}
   + \frac{d_1}{\sqrt{t_1^2+d_1^2}} \frac{\partial}{\partial d_1}.
\end{equation*}
Then,
\begin{align}
\nonumber
 \left(\Lambda_t\right)^{\ast} =&-\Lambda_t -\frac{\partial}{\partial t_1}
  \left( \frac{t_1}{\sqrt{t_1^2+d_1^2}}\right)
   -\frac{\partial}{\partial d_1}
  \left( \frac{d_1}{\sqrt{t_1^2+d_1^2}}\right) \\
 \nonumber
   =& -\Lambda_t -\frac{1}{\sqrt{t_1^2+d_1^2}} \\
  \label{adjlamt}
 =& -\Lambda_t-\frac{1}{|z_1|}.
\end{align}
Furthermore, from the relation
 \eqref{coortd}, we can write
\begin{equation*}
\frac{\partial}{\partial \overline{z}_1} =
 \alpha_1 \frac{\partial}{\partial z_1}
  +\alpha_2 \frac{\partial}{\partial t_1},
\end{equation*}
where $\alpha_1(|z_2|)$ and $\alpha_2(|z_2|)$ are bounded away
from 0 and depend smoothly on $|z_2|$.

We recall that
 that $g(\xi,z_1)$ has the property
$ \Lambda_tg(\xi,z_1) =0,$ and so
\begin{align*}
 g(\zeta-ik,z_1)
(\Lambda_t)^k z_1^{-i\zeta/2-1}\overline{z}_1^{-i\zeta/2} =&
 (\Lambda_t)^k\left[ z_1^{-i\zeta/2-1}\overline{z}_1^{-i\zeta/2}
 g(\zeta-ik,z_1)\right].
\end{align*}

 We thus have, after commuting the $z$ derivatives with the
 $\zeta$ integrals,
\begin{align}
\nonumber \int_{D_{\beta}}
   &\frac{\partial^k}{\partial w_1^k}
   \overline{\mathcal{T}_{-1}(z,w)} f(z) dV(z)=\\
\nonumber
 -&\frac{1}{2\pi}
\int_{D_{\beta}} \frac{1}{\overline{z}_2 w_2}\times\\
\nonumber & \qquad
 \Bigg[\frac{1}{\pi}
 \frac{\partial^k}{\partial {\overline{z}_1}^k} \int_{\mathbb{R}}
 \frac{(\zeta-ik)\zeta}
   {\sinh[(2\beta-\pi)(\zeta-ik)]
   \sinh(\pi\zeta)}
   \overline{z}_1^{-i\zeta-1}w_1^{i\zeta-1} d\zeta+\\
 \label{twoints}
   &
(\Lambda_t)^k \int_{\mathbb{R}}
 \frac{\zeta}{\sinh(\pi\zeta)}g(\zeta-ik,z_1)
 z_1^{-i\zeta/2-1}\overline{z}_1^{-i\zeta/2}
  w_1^{i\zeta-1}
  d\zeta
   \Bigg] f_{-1}(z) dV(z).
\end{align}

Integrating by parts in the first integral on the right in
(\ref{twoints}) gives
\begin{align}
\nonumber &  -\frac{1}{2\pi^2} \int_{D_{\beta}}
 \frac{1}{\overline{z}_2 w_2}
\times\\  \nonumber  &\quad \left[ \frac{\partial^k}{\partial
{\overline{z}_1}^k}\int_{\mathbb{R}}
 \frac{(\zeta-ik)\zeta}
 {\sinh[(2\beta-\pi)(\zeta-ik)]
   \sinh(\pi\zeta)}
   \overline{z}_1^{-i\zeta-1}w_1^{i\zeta-1}
  d\zeta\right] f_{-1}(z) dV(z) \notag \\
  \nonumber
&=   -\frac{1}{2\pi^2} \int_{D_{\beta}}
 \frac{1}{\overline{z}_2 w_2}
\times\\
  \nonumber  &\quad \left[ (\alpha_2\partial_{t_1})^k \int_{\mathbb{R}}
 \frac{(\zeta-ik)\zeta}
 {\sinh[(2\beta-\pi)(\zeta-ik)]
   \sinh(\pi\zeta)}
   \overline{z}_1^{-i\zeta-1}w_1^{i\zeta-1}
  d\zeta\right] f_{-1}(z) dV(z) \nonumber \\
 &
=(-1)^{k+1} \frac{1}{2\pi^2} \int_{D_{\beta}} \frac{1}{
\overline{z}_2 w_2}
 \times
\nonumber
 \\
 &\qquad
  \left[\int_{\mathbb{R}}
 \frac{(\zeta-ik)\zeta}
 {\sinh[(2\beta-\pi)(\zeta-ik)]
   \sinh(\pi\zeta)}
   \overline{z}_1^{-i\zeta-1}w_1^{i\zeta-1}
  d\zeta\right] (\alpha_2\partial_{t_1})^k f_{-1}(z) dV(z). \label{l21}
\end{align}

Similarly, we perform an integration by parts in the second
integral in (\ref{twoints}),  using \eqref{adjlamt}.
\begin{align}
\nonumber
 -&\frac{1}{2\pi}\kern-2pt\int_{D_{\beta}}\kern-4pt
  \frac{1}{\overline{z}_2 w_2} (\Lambda_t)^k \kern-2pt\int_{\mathbb{R}}
 \frac{\zeta}{\sinh(\pi\zeta)}g(\zeta-ik,z_1)
 z_1^{-i\zeta/2-1}\overline{z}_1^{-i\zeta/2}
  w_1^{i\zeta-1}
  d\zeta
   f_{-1}(z) dV(z)\\
 \nonumber
 &= (-1)^{k+1}
  \frac{1}{2\pi}
  \int_{D_{\beta}}
 \frac{1}{\overline{z}_2 w_2}
  \Bigg[
 \int_{\mathbb{R}}
 \frac{\zeta}{\sinh(\pi\zeta)}g(\zeta-ik,z_1)
 z_1^{-i\zeta/2-1}\overline{z}_1^{-i\zeta/2}
  w_1^{i\zeta-1}
  d\zeta\Bigg]\times\\
  &\qquad\qquad \left (\Lambda_t + |z_1|^{-1}\right)^k
  f_{-1}(z)dV(z)
  \label{l22}.
\end{align}

To finish the proof we note that
 the proof of Proposition \ref{mappingprop} , with
$\hat{h}(\xi,y) e^{-ix\xi}$ replaced with
\begin{equation*}
\frac{(\xi-ik)\xi}
 {\sinh[(2\beta-\pi)(\xi-ik)]
   \sinh(\pi\xi)}
   e^{-i\overline{z}_1\xi}
\end{equation*}
 may be followed
 to show that
 the operator from (\ref{l21}) with kernel
 \begin{equation*}
\int_{D_{\beta}}
 \frac{1}{\overline{z}_2 w_2}
\int_{\mathbb{R}}
 \frac{(\zeta-ik)\zeta}
 {\sinh[(2\beta-\pi)(\zeta-ik)]
   \sinh(\pi\zeta)}
   \overline{z}_1^{-i\zeta-1}w_1^{i\zeta-1}
  d\zeta
\end{equation*}
 maps
 $L^2(D_{\beta})$ to $L^2_{-1}(D_{\beta})$.
 Similarly, the proof of Proposition \ref{mappingprop} shows that
 the operator with kernel
\begin{equation*}
\frac{1}{2\pi} \frac{1}{\overline{z}_2 w_2}
 \int_{\mathbb{R}}
 \frac{\zeta}{\sinh(\pi\zeta)}g(\zeta-ik,z_1)
 z_1^{-i\zeta/2-1}\overline{z}_1^{-i\zeta/2}
  w_1^{i\zeta-1}
  d\zeta
 \end{equation*}
 occurring in (\ref{l22})
maps $L^2(D_{\beta})$ to $L^2_{-1}(D_{\beta})$. 
 We have estimates for the term in \eqref{l22} when $f\in
 W_{\mathscr{D}}^s(D_{\beta})$:
\begin{align*}
 \sum_{\alpha \le k} \left\| |z_1|^{-k+\alpha} \Lambda_t^{\alpha}
 f_{-1}\right.&\left.\right\|_{L^2(D_{\beta})}\\
  \lesssim& \sum_{\alpha \le k} \left\| |z_1|^{-k+\alpha} \Lambda_t^{\alpha}
 f_{-1}\right\|_{L^2((\mathbb{D}_1\times\mathbb{C})\cap D_{\beta})} +
 \|f_{-1}\|_{W^k(D_{\beta})}\\
\lesssim& \sum_{\alpha \le k} \left\| t_1^{-k+\alpha}
\Lambda_t^{\alpha}
 f_{-1}\right\|_{L^2((\mathbb{D}_1\times\mathbb{C})\cap D_{\beta})} +
 \|f_{-1}\|_{W^k(D_{\beta})}\\
 \lesssim&
 \|f_{-1}\|_{W^k(D_{\beta})},
\end{align*}
where $\mathbb{D}_1:= \{|z_1|\le 1\}$, the variable $t_1$ is as in
 \eqref{coortd},
and the last step follows from Theorem 1.4.4.4 in \cite{Gr} 
(with a slight variation in the argument we can also
 apply Theorem 11.8 in \cite{LiMa} which holds for smooth domains).

 Then,
 together (\ref{l21}) and (\ref{l22}) show
\begin{align*}
\left\| \frac{\partial^k}{\partial w_1^k} {\bf T}_{-1} f
\right\|_{L^2(D_{\beta})} \lesssim& \|f_{-1}\|_{W^k(D_{\beta})}
 +  \sum_{\alpha \le k} \left\| |z_1|^{-k+\alpha} \Lambda_t^{\alpha}
 f_{-1}\right\|_{L^2(D_{\beta})}
\\
\lesssim& \|f_{-1}\|_{W^k(D_{\beta})}.
\end{align*}
 
The estimate in (\ref{w1der}) is verified, completing the proof of
the theorem.
\end{proof}

\section{The case $j\ne -1$}
We construct operators
\begin{equation*}
{\bf T}_{j}: W^k_{\mathscr{D}}(D_{\beta}) \rightarrow
W^k_{j}(D_{\beta})
 \qquad \forall k,
\end{equation*}
for the cases $j\ne -1$ as follows.

We let $Q_j$ be the projection from
 $L^2(D_{\beta})$ to $L^2_j(D_{\beta})$ given by
\begin{equation*}
 Q_jf (z_1,z_2)=
  \frac{1}{2\pi} \int_{-\pi}^{\pi}
   f(z_1, e^{i\theta}z_2)e^{-ij\theta}d\theta.
 \end{equation*}
Then we take the operator ${\bf T}_j$ to be given by
\begin{equation*}
{\bf T}_j f= w_2^{j+1} {\bf T}_{-1} (z_2^{-j-1} Q_j f).
\end{equation*}
For each ${\bf T}_j$, due to properties of the operator ${\bf
T}_{-1}$, we have a theorem similar to
 Theorem \ref{projthrm}:
\begin{thrm}
 \label{tj}
 Let $\beta>\pi/2$, and $D_{\beta}$ be defined as above.
  For all $ j\in \mathbb{Z}$ there exists a bounded linear
  projection
\begin{equation*}
{\bf T}_j: L^2(D_{\beta}) \rightarrow B_{j}(D_{\beta})
\end{equation*}
which satisfies
\begin{equation*}
{\bf T}_j: W^k_{\mathscr{D}}(D_{\beta}) \rightarrow
W^k_{j}(D_{\beta})
 \qquad \forall k,
\end{equation*}
and
\begin{equation*}
\|{\bf T}_{j}f\|_{W^k_{j}(D_{\beta})}
 \lesssim \|f\|_{W^k(D_{\beta})}.
\end{equation*}
\end{thrm}

This proves the Main Theorem.

\section{Remarks}
We end with a few remarks.  We first note that in our proof of
Theorem \ref{projthrm}, we worked with Sobolev spaces, $W^k$ for
integer $k$.  The general case for all $s\ge 0$ follows by
interpolation.

Secondly, there are infinitely many projection operators which
have the same regularity properties as our constructed projection
in the Main Theorem.  Other projections can be constructed for
instance by changing the factor $\tau_{k}(\xi)$ in Section
\ref{correct}
 with the replacement of the term $e^{-\xi^2}$ with another
 $e^{-m\xi^2}$ for any positive integer $m$.  Then the rest of the
 arguments could be followed verbatim.

 By \cite{Ba95}, if the Bergman projection were to map
 $C^{\infty}_0 (\overline{D_{\beta}})$ continuously into
   $C^{\infty} (\overline{D_{\beta}})$ (it does not)
   then we would automatically have continuity from the larger
   space $C^{\infty} (\overline{D_{\beta}})$ as well.  Thus,
 it would be of interest to find an improvement to the
 projection,
 along the lines presented here,
 which preserves $W^s_j(D_{\beta})$ for all $s\ge 0$.

We lastly note that, while it would be ideal to obtain an
 operator which would map $W^s$ to itself, without the restriction
 to the space $W^s_j$, by summing the operators in Main Theorem
 \ref{mainproj} over $j$, the dependence of the norms in
 Theorem \ref{tj} on $j$ prohibit
 the convergence of such a summation.
Following the calculations of the proof of Proposition
\ref{mappingprop}
 leads to the estimates for the norms of ${\bf T}_j$:
\begin{equation*}
\|T_j\|\lesssim \frac{\sinh[(j+1)(\beta-\pi/2)]}{j+1}.
\end{equation*}

This exponential growth of the estimates thus prohibits us from
using results such as the Cotlar--Stein almost orthogonality lemma
to conclude any
 convergence of a sum over the operators ${\bf T}_j$.

\end{document}